\documentclass[10 pt, conference]{ieeeconf}
\pdfoutput=1

\IEEEoverridecommandlockouts                              
\usepackage{bbold}
\usepackage{color}
\usepackage{graphics} 
\usepackage{needspace}

\input{mysymbol.sty}
\usepackage{theorem}
\usepackage{cite}
\usepackage{amsmath}
\usepackage{amssymb}
\usepackage{mathtools}
\usepackage{multirow}
\usepackage{url}
\usepackage{graphics, subfigure, times, amsfonts}
\usepackage{tikz, epic,eepic}
\usetikzlibrary{shapes,arrows}
\usepackage{pgfplots}
\usepackage{color}
\usepackage{hyperref}
\usepackage{epstopdf}
\usepackage{epsfig, amsbsy}
\usepackage{latexsym}
\usepackage{amscd, verbatim}
\usepackage{multirow}
\usepackage{booktabs}

\usepackage{algorithm}
\usepackage{algorithmic}

\renewcommand{\baselinestretch}{0.97}\selectfont
\newtheorem{theorem}{Theorem}

\newtheorem{corollary}{Corollary}
\newtheorem{lemma}{Lemma}
\newtheorem{definition}{Definition}

{\itshape}{\rmfamily}
\newtheorem{assumption}{Assumption}
\newtheorem{remark}{Remark}

\def\forall{\text{for all\ }}

\title{\LARGE 
{Decentralized Fictitious Play in Near-Potential Games with Time-Varying Communication Networks}}
\author{Sarper Ayd\i n, Sina Arefizadeh and Ceyhun Eksin %
\thanks{S. Aydin and C. Eksin are with the Industrial and Systems Engineering Department, Texas A\&M University, College Station, TX 77843. E-mail:{\tt\small  \; sarper.aydin@tamu.edu; eksinc@tamu.edu.} This work was supported by NSF CCF-2008855.
}}

\begin{document}
\normalsize
\maketitle

%
\begin{abstract}
We study the convergence properties of decentralized fictitious play (DFP) for the class of near-potential games where the incentives of agents are nearly aligned with a potential function. In DFP, agents share information only with their current neighbors in a sequence of time-varying networks, keep estimates of other agents' empirical frequencies, and take actions to maximize their expected utility functions computed with respect to the estimated empirical frequencies. We show that empirical frequencies of actions converge to a set of strategies with potential function values that are larger than the potential function values obtained by approximate Nash equilibria of the closest potential game. This result establishes that DFP has identical convergence guarantees in near-potential games as the standard fictitious play in which agents observe the past actions of all the other agents. 
\end{abstract}

%

%
\section{Introduction}



A game comprises of multiple agents taking actions to maximize their individual utility functions. Potential games is a special class of game in which there exists a common potential function that captures the incentives of all the agents \cite{monderer1996potential}. Potential games are used to model behavior of agents in multitude of problems including traffic routing problems in transportation systems \cite{marden2009joint}, power allocation in cognitive radio \cite{perlaza2011learning}, task scheduling in robotics \cite{eksin2017distributed}. Many distributed game-theoretic learning dynamics, e.g., best response \cite{monderer1996potential}, fictitious play (FP) \cite{monderer1996fictitious}, converge to a Nash equilibrium of any potential game. Nash equilibrium is an action profile in which no agent would benefit by unilaterally deviating from. The convergence of learning dynamics to a NE provides a justification for modeling agents as rational in strategic environments. That is, we observe individuals behaving rationally in reality despite lacking the acuteness, because they learned to act rational over time as they repeatedly interacted with each other in the same environment. Moreover, in technological systems, agents (sensors, robots) can be coded to implement these learning dynamics so that the system reaches a desired performance. 

However, the individual utility functions may deviate from the potential function due to unknown parameters, or the inability of agents to accurately compute the learning rules. Still, we would expect the learning dynamics to reach close to a Nash equilibrium of the potential game when such deviations are small. The class of near-potential games, proposed in \cite{candogan2011flows,candogan2013dynamics}, quantifies the magnitude of such deviations under which desirable convergence properties continue to hold. Specifically, agents in a near-potential game have utility functions where the difference in utility values between any pair of actions for any agent is close enough to the differences in potential function values of a potential game, so that the game-theoretic learning dynamics converge to a neighborhood of a NE of the potential game. In \cite{candogan2013dynamics}, the size of the convergence neighborhood is characterized as a function of the closeness of games for common distributed learning algorithms (best response, and FP). Here, we characterize the convergence neighborhood based on the closeness of a potential game when agents' actions follow a decentralized version of the FP, i.e., the DFP algorithm.

Agent updates in distributed learning dynamics, e.g., best response, and FP, rely on information from all the agents. In FP, agents assume that other agents pick their actions according to a stationary distribution that is the empirical frequency of their past actions \cite{young2004strategic}. Given this assumption, agents take actions to maximize their expected utilities computed with respect to these empirical frequencies. Such a response is not possible, when agents do not have access to the past actions of all the other agents, e.g., when agents can only communicate with a subset of the  agents at each step. In such a scenario, agents can only keep estimates of the empirical frequencies of the other agents, and take the action that maximizes their expected utility functions computed with respect to the {\it estimated} empirical frequencies. The DFP, developed in \cite{eksin2017distributed}, is one such learning algorithm that is fitting to communication over time-varying networks.

Given the DFP dynamics, we show that if the empirical frequencies are far from any equilibrium, i.e., are outside the set of approximate $\epsilon$-Nash equilibria of the near-potential game where the constant $\epsilon$ depends on the closeness of games, then the potential function value of the close potential game increases at every action update (Lemma \ref{lem_inc}). We note that the set of approximate $\epsilon$-Nash equilibria represents strategies in which an agent's incentive to deviate is not larger than $\epsilon$. Leveraging this result, we show that the empirical frequencies converge to a set of mixed strategies (a distribution over the actions) with potential function values larger than the lowest potential function value obtained by strategies in the set of  $\epsilon$-Nash equilibria (Theorem \ref{thm_min} and Corollary \ref{cor_1}). The closeness conditions required and the bounds established in these results are identical to the ones for FP dynamics shown in \cite{candogan2013dynamics}. That is, DFP has the same convergence guarantees as FP in near-potential games, despite the fact that agents   communicate only with a subset of the agents at each step. We validate our results in a target assignment problem in which agents are uncertain about the target locations. Numerical experiments show convergence of actions to a NE of the game with known target locations when the near-potential game is close enough to the original game, i.e., when the uncertainty is small.

\section{Near-Potential Games}\label{sec:model}
A game consists of $N$ agents represented by the set $\ccalN:=\{1,\dots,N\}$. Each agent $i\in\ccalN$ selects an action $a_i$ over a finite set of actions $\ccalA_i$ to maximize its utility (payoff) $u_i:\ccalA_i \times \ccalA_{-i} \to \reals$ where $\ccalA_{-i}:= \prod_{i \neq j} \ccalA_i$. A game $\Gamma$ is a tuple of agents $\ccalN$, action space $\ccalA^N:=\prod_{i \in \ccalN} \ccalA_i$, and utility functions $u_i(\cdot)$ for $i \in \ccalN$. 

We define the class of potential games \cite{monderer1996potential} in the following. 
\begin{definition}[Potential Games] \label{def_potential}
A game $\Gamma$ is a potential game, if there exists a function $u: \ccalA^N \rightarrow \mathbb{R}$ such that the following relation holds for all agents $i \in \ccalN$,
\begin{equation}\label{eq_def_potential}
    u(a'_i,a_{-i})-u(a_i,a_{-i})= u_i(a'_i,a_{-i})-u_i(a_i,a_{-i})
\end{equation}
where $a_i' \in \ccalA_i$ and $a_i\in \ccalA_i$ and $a_{-i} \in \ccalA_{-i}$. The corresponding function $u: \ccalA^N \rightarrow \mathbb{R}$ is called the  potential function of the game $\Gamma$. 
\end{definition}
The potential function ($u(\cdot)$)  captures the difference in individual payoffs as a result of unilateral deviation in action for every agent. The existence of a potential function assures the existence of a NE in potential games. Furthermore, every maximizer of the potential function is a NE of the potential game \cite{monderer1996fictitious}.

In order to define the class of near-potential games, we define a notion of closeness between games.

\begin{definition}[Maximum Pairwise Difference] \label{def_mpd}
Let $\Gamma=(\ccalN,\{\mathcal{A}_i,u_i\}_{i\in\ccalN})$ and  $\hat{\Gamma}=(\ccalN,\{\mathcal{A}_i,\hat{u}_i\}_{i\in\ccalN})$ be two  games with the same set of agents $\ccalN$ and action sets $\{\mathcal{A}_i\}_{i \in \ccalN}$ and, utilities $\{u_i\}_{i \in \ccalN}$ and $\{\hat{u}_i\}_{i \in \ccalN}$. Further, let $d_{(a'_i,a)}^{\Gamma}:=u_i(a'_i,a_{-i})-u_i(a_i,a_{-i})$ be a difference in utility of an agent $i$ by unilateral change to an action $a'_i \in \ccalA_i$, given joint action profile $a=(a_i,a_{-i}) \in \ccalA^N$ in the game $\Gamma$. Then, the maximum pairwise difference $d(\Gamma, \hat{\Gamma})$ between the games $\Gamma$ and $\hat{\Gamma}$ is defined as, 

\begin{equation}\label{eq_mpd}
    d(\Gamma, \hat{\Gamma}):=  \underset{i \in \ccalN, \, a'_i \in \ccalA_i,\, a  \in \ccalA^N }{\max} |d_{(a'_i,a)}^{\Gamma}- d_{(a_i',a)}^{\hat{\Gamma}}|.
\end{equation}
\end{definition}

The maximum pairwise difference (MPD), introduced by \cite{candogan2013dynamics}, defines the distance between two games based on the difference in agent payoffs resulting from unilateral changes to agent actions. A near-potential game is a game that is {\it close} to a potential game in terms of MPD. 

\begin{definition} [Near-Potential Games]
A game $\Gamma$ is a near-potential game if there exists a potential game $\hat{\Gamma}$ within a maximum-pairwise distance (MPD), $d(\Gamma, \hat{\Gamma}) \le \delta$ where $\delta \in \mathbb{R}^+$.
\end{definition}

The class of near potential games relaxes the condition of potential games in  \eqref{eq_def_potential} similar to ordinal, weighted, or best-response potential games. In this paper, we consider decentralized learning dynamics in $\delta$ near-potential games. We assume the nearest potential game and its potential function is known. In general, the closeness of a game to a potential game can be determined by solving a convex optimization problem---see  \cite{candogan2011flows} for details.


%

\section{Decentralized Fictitious Play}\label{sec:DFP}
%
%
%


Fictitious play is a game-theoretic learning mechanism in which each agent repeatedly takes an action to maximize its payoff based on the estimates of other agents' strategies. Agent $i$ forms an estimate of agent $j$'s strategy by keeping an empirical frequency of its past actions. 

To formally define how these estimates are formed we need the following definitions. For simplicity, we assume the action space of agents is common, i.e., $\ccalA_i = \ccalA$ where the number of possible actions is equal to $K \in \naturals^+$, i.e., $|\ccalA|=K$. 
We let $\Delta \ccalA$ be the probability space over the common action space. The strategy of agent $i$ is a distribution on the action space denoted as $\sigma_i \in \Delta\ccalA$, where $\sigma_{i}(a_i)\in [0,1]$ denotes the probability of selecting action $a_i\in \ccalA$. We define the expected utility $u_i: \Delta\ccalA^N \to \reals$ as 
\begin{align} \label{eq_ut_mix}
    u(\sigma_i,\sigma_{-i})
    &=\sum_{a \in \ccalA^{N}} u(a_i,a_{-i}) \sigma(a),
\end{align}
where $\sigma = (\sigma_i,\sigma_{-i}) \in \Delta\ccalA^N$ is the joint strategy profile. 

In standard fictitious play, agents repeatedly take actions in discrete time steps $t=1,2,\dots$. Each agent determines its next action $\sigma_{i,t}\in \Delta\ccalA$ to maximize its expected utility assuming other agents play according to a stationary distribution $f_{-i,t}:= \{ f_{j,t}\}_{j \in \ccalN\setminus i}$ where $f_{j,t} \in \Delta \ccalA$, i.e., 

\begin{equation} \label{eq_br_naive}
    {\sigma_{i,t}} \in \argmax_{{\sigma_i \in \Delta\ccalA}} u_i({\sigma_{i}},f_{-i,t}).
\end{equation}
The stationary distribution $f_{i,t}\in \Delta \ccalA$ is computed using the past empirical frequency of actions, $f_{i,t}= \frac{1}{t} \sum_{\tau=1}^t  {\sigma_{i,\tau}}$, which can equivalently be written as
\begin{equation}\label{eq_empirical_frequency}
f_{i,t}=\frac{t-1}{t}f_{i,t-1}+\frac{1}{t}{\sigma_{i,t}}.
\end{equation}
%
The recursive form above allows each agent $i$ to compute the empirical frequencies of other agents $j \in \ccalN \setminus i$ by only keeping the past empirical frequencies $\{f_{j,t-1}\}_{j\in \ccalN \setminus i}$ in memory and observing their current actions $\{\sigma_{j,t}\}_{j \in \ccalN \setminus i}$.

Here we consider a scenario where agents communicate over a {time-varying} network $\ccalG_t=(\ccalN,{\ccalE_t})$, where ${\ccalE_t}$ represents the edge set, determining the set of agents each agent can communicate with, {at each time step $t$}. In this setting, agent $i$ can only communicate with its neighbors ${\ccalN_{i,t} := \{j: (i,j)\in {\ccalE_t}\}}$. Thus, it cannot compute the empirical frequency of all the other agents as in \eqref{eq_empirical_frequency}. Instead,  we replace the actual empirical frequencies $f_{j,t}\in \Delta \ccalA$ with a local copy $\upsilon_{j,t}^i \in {\Delta\ccalA}$ kept at agent $i$. This local copy is the estimate of agent $j$'s empirical frequency at time $t$ by agent $i$. We let agent $i$'s local copy of its own empirical frequency be equal to its empirical frequency,  $\upsilon_{i,t}^i = f_{i,t}$.  In each step, agent $i$ updates its local estimate by using the local copies shared by its neighbors,
\begin{equation}\label{eq_local_update}
    \upsilon_{j,t}^i= \sum_{l \in \ccalN} {w_{jl,t}^i} \upsilon_{j,t}^{l},
\end{equation}
where $w_{jl,t}^i \ge 0$ is the weight that agent $i$ puts on agent $l$'s estimate of agent $j$ {at time t}.


In decentralized fictitious play, each agent takes an action that maximizes its expected utility (best-respond) computed using its estimates of others' empirical frequencies $\upsilon_{-i,t}^i:=\{\upsilon_{j,t-1}^{i}\}_{j\in\ccalN\setminus i}$,
\begin{equation}\label{eq_br_local}
    {\sigma_{i,t}} \in \arg \max_{  {\sigma_i \in \Delta\ccalA}} u_i({\sigma_{i}}, \upsilon_{-i,t}^i).
\end{equation}
%
An outline of the DFP algorithm is given below.

\begin{algorithm}[H] 
   \caption{DFP for Agent $i$}
\label{suboptimal_alg_inner}
\begin{algorithmic}[1]\label{alg_DFP}
   \STATE {\bfseries Input:} Local estimates $\upsilon_{-i0}^i$ and networks $\{\ccalG_t\}_{t\geq 1}$.
\FOR{$t=1,2,\cdots $} 
    \STATE Share local copies $\{\upsilon^i_{j,t}\}_{j\in\ccalN}$ with $l\in \ccalN_{i,t}$
    \STATE Update local copies $\upsilon_{j,t}^i$ \eqref{eq_local_update}.
    \STATE Determine action $\sigma_{i,t}$ \eqref{eq_br_local} and  update $f_{i,t}$ \eqref{eq_empirical_frequency}. 

  \ENDFOR 
   \end{algorithmic}
\end{algorithm}

 \section{Convergence of DFP in Near-Potential Games} \label{sec::conv}

\subsection{Game theoretic preliminaries}
Given a game $\Gamma:=\{\ccalN, \ccalA^N, \{u_i\}_{i\in\ccalN}\}$, a Nash equilibrium (NE) defines a strategy profile such that no agent can individually increase its utility by deviating to another strategy. Given a strategy profile $\sigma^*$, when there exists a profitable deviation for an agent but the increase in utility from this deviation is no more than $\epsilon\geq 0$, the strategy profile $\sigma^*$ is said to be an approximate NE. 

\begin{definition} [Approximate Nash Equilibrium] \label{def_app_Nash}
The strategy profile  $\sigma^*=(\sigma_i^*,\sigma_{-i}^*) \in \Delta \ccalA^N$ is an $\epsilon$-Nash equilibrium of the game $\Gamma$ for some $\epsilon\geq 0$ if and only if for all $i\in\ccalN$
\begin{equation} \label{eq_app_Nash}
    u_i(\sigma^*_i,\sigma^*_{-i}) - u_i(\sigma_i,\sigma_{-i}^*) \ge -\epsilon, \quad \forall \sigma_i \in \Delta\ccalA.
\end{equation}
The strategy profile $\sigma^*$ is called a pure $\epsilon$-NE if the strategy selects a single action profile $a^* = (a_i^*, a_{-i}^*)\in \ccalA^N$ with probability one. If the $\epsilon$-NE strategy profile is not pure, then it is called an $\epsilon$-mixed NE strategy profile. 
\end{definition}
We denote the set of $\epsilon$-Nash equilibria in a game $\Gamma$ using $\Delta_\epsilon$. We recover the definition of a NE strategy profile when $\epsilon=0$ in the above definition. Using the same notation, we denote the set of Nash equilibria by $\Delta_0$. 

\subsection{Convergence}

We make the following set of assumptions for the time-varying communication network.
\begin{assumption}\label{as_connect}
The network $\ccalG=(\ccalN,\ccalE_{\infty})$ is connected, where $\ccalE_{\infty}=\{(i,j)| (i,j) \in \ccalE_t, \, \text{for infinitely many t} \in \naturals \}$.
\end{assumption}
This assumption states that starting from any time $t_0$, there exists a path from agent $j$ to $i$ for any pair of agents $i$ and $j$ when we consider the edges $\bigcup_{t\geq t_0} \ccalE_{t}$.
\begin{assumption}\label{as_bounded_com}
There exists a time step $T_{B}>0$, such that for any edge $(i,j) \in \ccalE_{\infty}$ and $t \ge 1$, it holds  $(i,j) \in \bigcup_{\tau=0}^{T_{B}-1}\ccalE_{t+\tau}$.
\end{assumption}
This assumption means the edge $(i,j)\in \ccalE_{\infty}$ also belongs to the edge set $\bigcup_{t_0+T_B > t\geq t_0} \ccalE_{t}$ for any time $t_0>0$. Assumptions \ref{as_connect} and \ref{as_bounded_com} are referred to as {\it connectivity} and {\it bounded communication interval}, respectively in \cite{NedicOzdaglar}.
%
\begin{assumption}\label{as_com_weights}
There exists a scalar $0<\eta<1$, such that the following statements hold for all $j\in \ccalN$ and $i\in \ccalN$,
\begin{itemize}
    \item[\textit{(i)}] If $l \in \ccalN_{i,t} \cup  \{i\}$, then $w_{jl,t}^i \ge \eta$. Otherwise,  $w_{jl,t}^i=0$,
    \item[\textit{(ii)}] $w_{ii,t}^i =1$ for all $t$,
    \item [\textit{(iii)}] $\sum_{l \in \ccalN_{i,t} \cup \{i\} } w_{jl,t}^i=1$ for all $t$.
\end{itemize}
\end{assumption}
Assumption 3{\it (i)} makes sure that agents only put positive weight on their current neighbors' estimates in \eqref{eq_local_update}. Assumption 3{\it (ii)} ensures $\nu^i_{i,t}=f_{i,t}$ for all $t>0$. Assumption{\it(iii)} means that the weights matrix obtained by placing each agent's weights in a row is row stochastic for all times. 

Assumptions \ref{as_connect}-\ref{as_com_weights} ensure that local information stored by agent $i \in \ccalN$ reaches an agent $j \in \ccalN \setminus\{i\}$ in finite time. Next result provides a rate for the convergence of the local copies of empirical frequencies $\nu^i_{j,t}$ to the actual empirical empirical frequencies $f_{j,t}$---see \cite{arefizadeh2019distributed} for the proof. 
 
\begin{lemma}[Proposition 1,  \cite{arefizadeh2019distributed}]\label{lem_ups_to_fic}
Suppose Assumptions \ref{as_connect}-\ref{as_com_weights} hold. If $f_{j0}=\upsilon^i_{j,0}$ holds for all pairs of agents {$j\in \ccalN$ and $i\in \ccalN$}, then the local copies $\{\upsilon^i_{t}\}^{i \in \ccalN}_{ t \ge 0}$ converge to the empirical frequencies $\{f_{t}\}_{ t \ge 0}$ with rate $O(\log t /t )$, i.e., $|| \upsilon^i_{j,t} -f_{j,t} || = O(\log t /t )$ for all $j\in \ccalN$ and $i\in \ccalN$.
\end{lemma}
The proof relies on the properties of row stochastic matrices formed by the weights $\{w^i_{jl,t}\}_{l\in\ccalN}$. Note that we do not assume the weights form a doubly stochastic matrix, i.e., the sum of agents' weights for a given agent $j$ do not need to sum to one. This assumption would require agents to coordinate the weights they use to update their local estimates for a given agent $j$'s empirical frequency. 

Next, we assume the potential function is bounded. 


{\begin{assumption}\label{as_bounded_potential}
 The potential function $u: \Delta \ccalA^N \rightarrow \reals $ of the closest potential game $\hat{\Gamma}$ is bounded, where $d(\Gamma, \hat{\Gamma}) \le \delta$ .
\end{assumption}}

%
Using the above assumptions, we provide a lower bound for the change in the expected potential function values of the nearest potential game between consecutive time steps.

\begin{lemma}\label{lem_inc}
Suppose Assumptions \ref{as_connect}-\ref{as_bounded_potential} hold. Let $\Gamma$ be a $\delta$-near potential game for some $\delta\geq 0$. The potential function of the closest potential game is $u(\cdot)$. We denote the empirical frequency sequence generated by the DFP algorithm as $\{f_t\}_{t \ge 1}$. If the empirical frequency $f_t$ is outside the $\epsilon$-NE set for $\epsilon \ge 0$, then given a large enough $T>0$ it holds 
\begin{equation} \label{eq_step_inc}
    u({f_{t+1}})-u({f_{t}})\ge \frac{\epsilon-N\delta}{t+1} -O\Big(\frac{\log t}{t^2}\Big) \; \text{ for } t\geq T.
\end{equation}
\end{lemma}
\begin{proof}
Taylor's expansion of the expected utility yields 
\begin{align}\label{eq_taylor_preperation_1}
    u(f_{t+1})&- u(f_{t})=  \nonumber\\
    &\sum_{i=1}^N\sum_{a_{i}\in \ccalA} u(a_{i},f_{-i,t}) \big(f_{i,t+1}(a_i)-f_{i,t}(a_{i})\big)+\nonumber\\
    &O(||f_{i,t+1}-f_{i,t}||^2),
\end{align}
where $f_{i,t+1}(a_{i})$ denotes the probability of selecting action $a_i\in\ccalA$ according to the empirical frequency $f_{i,t+1}$. Using the empirical frequency updates \eqref{eq_empirical_frequency}, we have 
\begin{align}\label{eq_taylor_preperation_2}
    &u(f_{t+1})- u(f_{t})=  \nonumber\\
    &\sum_{i=1}^N\sum_{{a_{i}\in \ccalA}}\frac{1}{t+1} u(a_{i},f_{-i,t})\big({\sigma_{i,t}(a_{i})}-f_{i,t}(a_{i})\big)+ O\Big(\frac{1}{t^2}\Big). 
\end{align}
Since $u(f_{i,t},f_{-i,t})=\sum_{a_{i}\in \ccalA_i}u(a_{i},f_{-i,t}) f_{i,t}(a_{i})$ and { $u(\sigma_{i,t},f_{-i,t})=\sum_{a_{i}\in \ccalA_i}u(a_{i},f_{-i,t}) \sigma_{i,t}(a_{i})$}, we have
\begin{align}\label{eq_taylor_preperation_3}
    u(&f_{t+1})- u(f_{t})=  \nonumber\\
    &\frac{1}{t+1}\sum_{i=1}^N \big(u({\sigma_{i,t}},f_{-i,t})-u(f_{i,t},f_{-i,t})\big)+ O\Big(\frac{1}{t^2}\Big).
\end{align}
Using the fact that the game $\Gamma$ is a near-potential game, i.e., $d(\Gamma, \hat \Gamma)<\delta$, the above equality and Definition \ref{def_mpd} imply
\begin{align}\label{eq_taylor}
    &u(f_{t+1})- u(f_{t}) \ge  \nonumber\\
    &\frac{1}{t+1} \sum_{i=1}^N (u_i({\sigma_{i,t}}, f_{-i,t})- u_i(f_{i,t}, f_{-i,t})-\delta)+O\Big(\frac{1}{t^2}\Big).
\end{align}

{By the contrapositive statement of Lemma \ref{lem_sigma_ab} and via Lemma \ref{lem_ups_to_fic}, if $f_t \not \in {\Delta} _{\epsilon}$, then  there exists at least one agent $i^o \in \mathcal{N}^o \subseteq \ccalN$ where $|\ccalN^o| \ge 1$, whose  local estimates $\upsilon^{i_o}_{-i_o,t}$ are outside $(\epsilon-O(\log t /t))$-NE region for large enough $t$, i.e., $\upsilon^{i_o}_{-i_o,t}\not \in \Delta_{\epsilon-O(\log t /t)}$.} Hence, agent $i_0$'s utility value {changes} by at least $\epsilon-O(\log t /t)$,
\begin{equation}\label{eq_out}
    u_{i_o}({\sigma_{i_o, t}},\upsilon^{i_o}_{-i_o, t})-u_{i_o}(\upsilon_{i_o, t},\upsilon^{i_o}_{-i_o, t}) \ge \epsilon-O\Big(\frac{\log t}{t}\Big).
\end{equation}
Since each agent $i \in \mathcal{N} \setminus \ccalN^o$ best responds to local copies $\upsilon^i_{-i,t}$,
\begin{equation}\label{eq_br}
u_{i}({\sigma_{i, t}},\upsilon^{i}_{-i, t})-u_{i}(\upsilon_{i, t},\upsilon^{i}_{-i,t}) \ge 0.
\end{equation}
Using Lipschitz continuity of the mixed extension of the utility function and Lemma \ref{lem_ups_to_fic}, it holds again for all $i \in \ccalN$,
\begin{subequations}
\begin{align}
    &u_{i}({\sigma_{i, t}}, f_{-i, t})-  u_{i}({\sigma_{i, t}},\upsilon^{i}_{-i, t}) \ge -O\Big(\frac{\log t}{t}\Big), \label{eq_inc_1} \\
    &u_{i}(\upsilon_{i, t},\upsilon^{i}_{-i, t})-u_{i}(f_{i, t},f_{-i, t})\ge -O\Big(\frac{\log t}{t}\Big).\label{eq_inc_2}
\end{align}
\end{subequations}
Therefore, summing the left hand sides of \eqref{eq_inc_1}, \eqref{eq_inc_2}, with \eqref{eq_out} or {\eqref{eq_br}} for all agents $i \in \ccalN$ yields,
\begin{equation}\label{eq_sum_lower_1}
   \sum_{i=1}^N (u_i({\sigma_{i,t}}, f_{-i,t})- u_i(f_{i,t}, f_{-i,t})-\delta) \ge \epsilon-O\Big(\frac{\log t}{t}\Big)-N\delta.
\end{equation}
Thus, when \eqref{eq_sum_lower_1} is substituted into \eqref{eq_taylor}, we obtain the desired lower bound in \eqref{eq_step_inc}.
%
\end{proof}
The result above implies that if the empirical frequencies are outside an $\epsilon$-equilibrium for some $\epsilon>N \delta$ and time $t$ is large enough, then the potential function values will improve with each update of the algorithm in the near-potential game $\Gamma$. For the standard FP, the same improvement relation holds when we replace $O(\log(t)/t^2)$ term with $O(1/t^2)$---see Lemma 5.3 in \cite{candogan2013dynamics}. That is, the rate loss $O(\log(t))$ due to the local estimates trailing behind the actual empirical frequencies (Lemma \ref{lem_ups_to_fic}) appears as a slow down in the improvement of potential function values in \eqref{eq_step_inc}. 

Since the potential function is bounded within the space of actions, it cannot increase unboundedly which implies that the sequence will enter a set of strategies with potential function values comparable to that of approximate Nash equilibria. We formalize this intuition in the following result.
\begin{theorem}\label{thm_min}
Suppose Assumptions \ref{as_connect}-\ref{as_bounded_potential} hold. Let $\{f_t\}_{t \ge 1}$ be the sequence of empirical frequencies generated by Algorithm \ref{alg_DFP}. Then, for any $\epsilon>0$, there exists a time $T_{\epsilon}$ and $\delta>0$, such that for all $t>T_{\epsilon}$, the following holds, 
\begin{equation} \label{eq_near_eq}
    f_t \in C_{N\delta+\epsilon} := \{ \sigma \in \Delta \ccalA^N \: | \: u(\sigma) \ge \underset{y \in \Delta_{N\delta+\epsilon}}{\min} \, u(y)\}.
\end{equation}
\end{theorem}

\begin{proof}
The proof has two steps. First, we show that the region $\Delta_{N \delta +\epsilon'}$ for some  $\epsilon > \epsilon' >0$ is visited infinitely often. Assume that there exists a long enough time $\hat{T}$ such that for $t>\hat{T}$, it holds that $f_t \not \in \Delta_{N\delta+\epsilon'}$. The difference in potential function values between consecutive time-steps after long enough time $t>\hat{T}$ becomes,
\begin{align}\label{eq_inc_mod}
 u(f_{t+1})- u(f_{t})&\ge  \frac{1}{t+1}\Big(N\delta+\epsilon'-N\delta-O\Big(\frac{\log t}{t}\Big)\Big).\\
 &>\frac{\epsilon'}{2(t+1)}>0.
\end{align}
Further, summing over the consecutive time steps for all $t>\hat{T}$ provides 
\begin{equation}\label{eq_sup}
    \limsup_{t\rightarrow \infty} u(f_{t+1})- u(f_{\hat{T}})\ge \sum_{t=\hat{T}+1}^{\infty} \frac{\epsilon'}{2(t+1)}.
\end{equation}
Since the potential function $u$ is {bounded} by Assumption \ref{as_bounded_potential}, the left-hand side of the inequality above has to be finite, while the right-hand side is not. Hence, this contradicts our assumption yielding $\Delta_{N\delta+\epsilon'}$ is infinitely visited by $f_t $ for all $t>\hat{T}$, holds. 

\begin{figure*}
	\centering
	\begin{tabular}{ccc}
		\includegraphics[width=.3\linewidth]{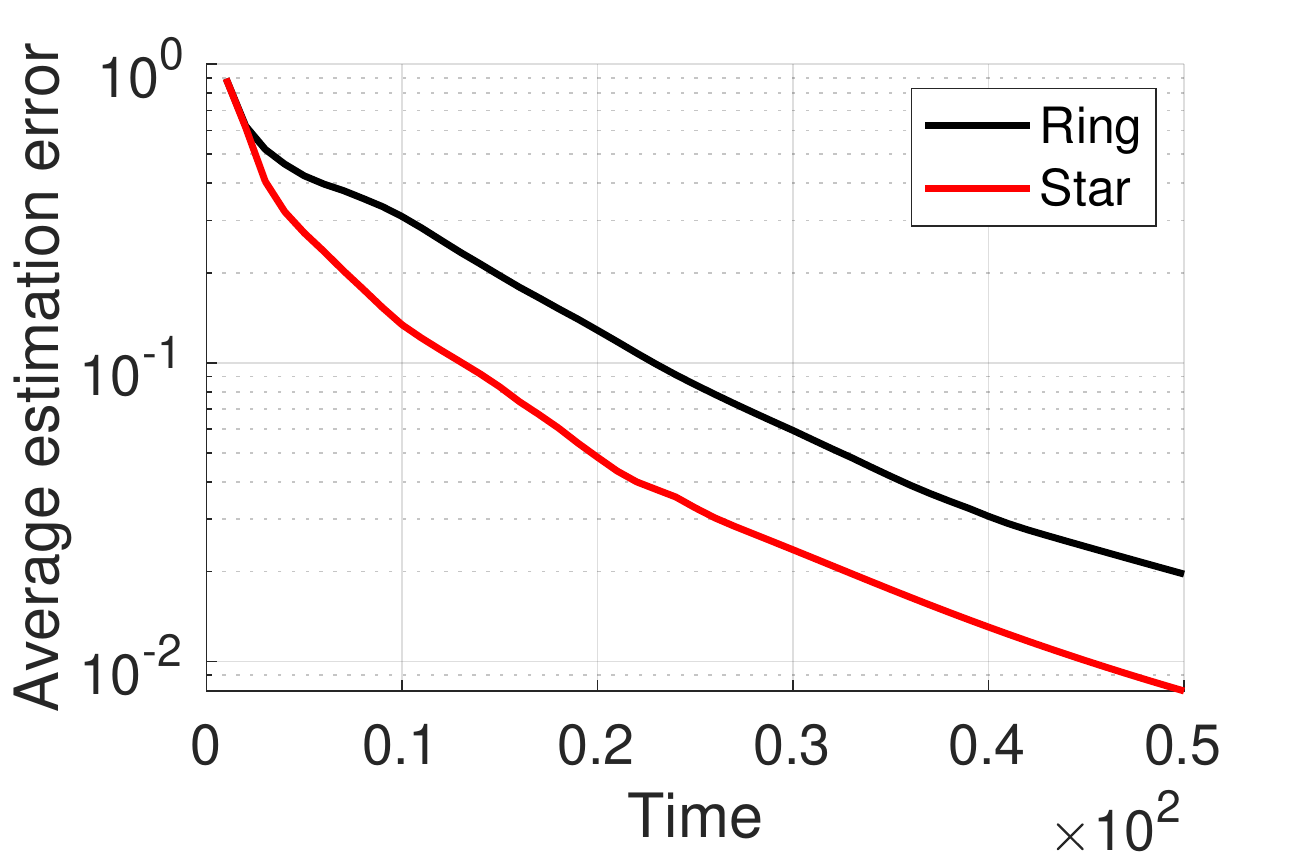}&
	\includegraphics[width=.3\linewidth]{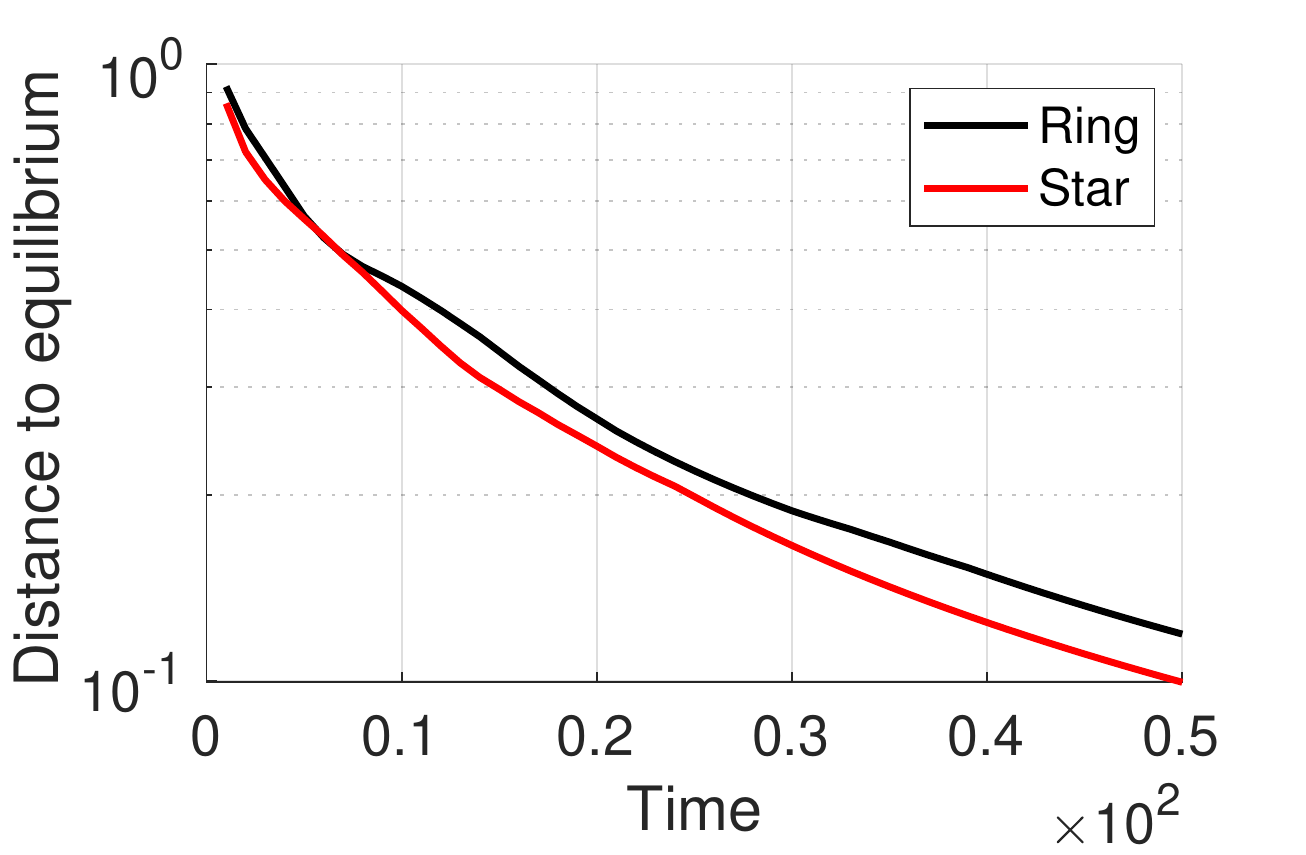}&
		\includegraphics[width=.3\linewidth]{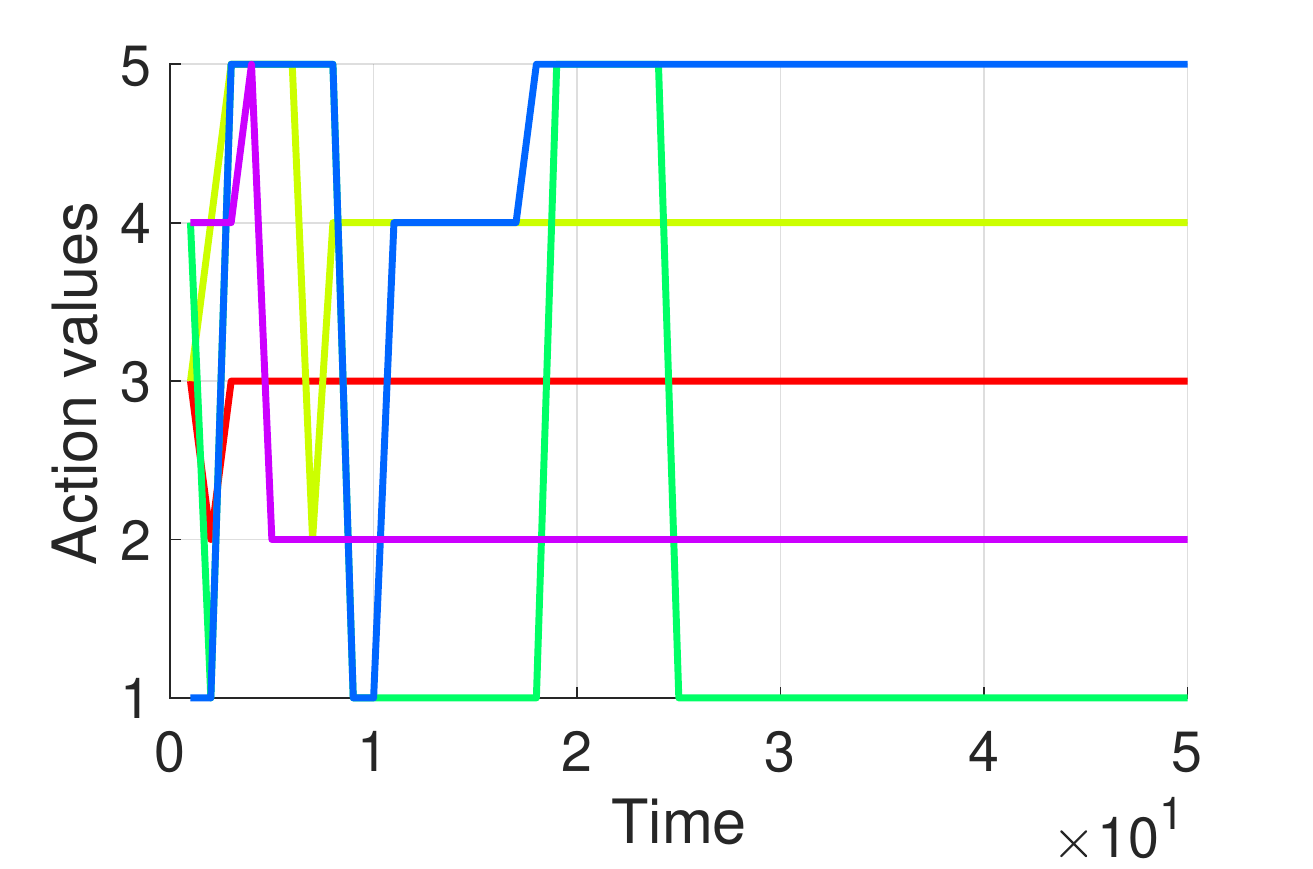} 
	\end{tabular}
	\vspace{-6pt}
	\caption{DFP in target assignment game with unknown target locations. (Left)Average estimation error $\frac{1}{N(N-1)}\sum_{i \in \ccalN} \sum_{j \in \ccalN \setminus \{i\}}||f_{it}-\upsilon_{it}^j ||$ over 20 runs. (Middle) Average distance to Nash equilibrium $\frac{1}{N}\sum_{i \in \ccalN} || f_{it}-a^*_i||$ over 20 runs.
 (Right) Action selections from a single run. } \vspace{-1pt}
	\label{fig_conv}
	\vspace{-12pt}
\end{figure*}

In the second part, we show that if $f_t\in C_{N\delta +\epsilon}$ at a large enough $t$, then $f_t$ will remain in $C_{N\delta +\epsilon}$. Consider the case that $f_t \in \Delta_{N\delta+\epsilon'}$. Observe that $||f_{t+1}-f_t||=O\Big(\frac{1}{t}\Big)$ as per \eqref{eq_empirical_frequency}. Thus, there exists some $T''>\hat{T}$ and for all $t>T''> \hat{T}$ such that we have
\begin{equation}\label{eq_eps}
    f_{t+1} \in \Delta_{N\delta+\epsilon}
\end{equation}
by Lemma \ref{lem_ups_to_fic} and the Lipschitz continuity of the mixed extension of the potential function. Next, we consider the case that $f_t \in  C_{N\delta+\epsilon} \setminus \Delta_{N\delta+\epsilon'}$. Using \eqref{eq_inc_mod} and the definition of $C_{N\delta+\epsilon}$, it holds that 
\begin{equation}\label{eq_eps_2}
    u(f_{t+1}) > u(f_t) \ge \underset{y \in \Delta_{N\delta+\epsilon}}{\min} u(y).
\end{equation}
Thus there exists some $T''>\hat{T}$ such that if $f_t \in  C_{N\delta +\epsilon}$ for $t\geq T''$, then $f_{t+1}\in  C_{N \delta +\epsilon}$.
\end{proof}

The following result shows that in the limit empirical frequencies converge to $C_{N \delta}$.

\begin{corollary}\label{cor_1}
Suppose Assumptions \ref{as_connect}-\ref{as_bounded_potential} hold. Let $\{f_t\}_{t \ge 1}$ be the sequence generated by Algorithm \ref{alg_DFP}. The empirical frequencies $\{f_t\}_{t\ge0}$, converge to the set given below,
\begin{equation}
    C_{N\delta} := \{ \sigma \in \Delta \ccalA^N \: | \: u(\sigma) \ge \underset{y \in \Delta_{N\delta}}{\min} \, u(y)\}.
\end{equation}
\end{corollary}
\begin{proof}
Consider the set $C_{N \delta + 1/q}$ for $q\in \mathbb{Z}^+$. By Theorem \ref{thm_min}, there exists a time $T_{q}$ such that for all $t> T_{q}$, $f_t \in C_{N\delta + 1/q}$ for any $q \in \mathbb{Z}^+$. Lemma \ref{lem_subset_c} states that for any $\xi>0$, $C_{N\delta + 1/q}\subset S_\xi$ for some $q\in \mathbb{Z}^+$, where $S_\xi$ is the $\xi>0$ neighborhood of $C_{N\delta}$ (see  \eqref{eq_xi_neighborhood} for a formal definition). Thus, we have that there exists a time $T_\xi$ such that for all $t> T_\xi$, $f_t\in S_\xi$ for any $\xi>0$ and some $q\in \mathbb{Z}^+$.
From the definition of $S_{\xi}$, the following holds,
\begin{equation}
    \limsup_{t \rightarrow \infty} \underset{z \in C_{N\delta} }{\min} ||z-f_t|| < \xi.
\end{equation}
Since, $\xi>0$ can be arbitrarily small and $||z-f_t|| \ge 0$, we can conclude that  $\underset{t \rightarrow \infty}{\lim }\: \underset{z \in C_{N\delta} }{\min} ||z-f_t||=0$.
\end{proof}

 The above corollary states that the joint empirical frequencies converge to a set of strategies that have potential values larger than the potential value of the strategy with the minimum potential value belonging to the approximate equilibrium set $\Delta_{N\delta}$. Note that it also implies convergence of DFP to a NE in exact potential games, i.e., when $\delta=0$, capturing the convergence results in \cite{Swenson_et_al_2014,eksin2017distributed}.

\begin{remark}
Our analysis shows that the empirical frequencies $\{f_t\}_{t\geq 1}$ generated by the DFP will be on par with potential function values of the strategies in the set of $N \delta$-equilibria of the closest potential game ($C_{N \delta}$). However, it is not clear whether the sequence $\{f_t\}_{t\geq 1}$ would remain within the neighborhood of a single Nash equilibrium or visit neighborhoods of multiple Nash equilibria that belong to the set $C_{N \delta}$. For standard FP, the sequence $\{f_t\}_{t\geq 1}$ converges to a neighborhood of a single NE given that any two Nash equilibria of the game are far enough from each other---see Theorem 5.2 in \cite{candogan2013dynamics}. We conjecture that $\{f_t\}_{t\geq 1}$ in DFP will also converge to a neighborhood of a single NE under the same assumption. 
\end{remark}
\section{Numerical Experiments}\label{sec::numeric}


We consider a target assignment problem with $N=5$ autonomous agents and $K=5$ targets with the objective is to cover all targets with minimum effort as a team. 
We represent this objective with the following utility function 
%
\begin{equation}\label{util_target}
    u_{i}(a_i=k,a_{-i})= \frac{\sum_{a_i \in \ccalA}a_i\bbone_{a_{-ik}=0}}{ \sum_{a_i \in \ccalA} a_i d_{ik}},
\end{equation}
{where $k\in \ccalA$, and $\bbone_{a_{-ik}=0} \in \{0,1\}$ is a binary value that it is equal to $1$ if none of the other agents $j \in \ccalN \setminus \{i\}$ select target $k$, and otherwise it is equal to 0.} The distance vector of agent $i$ is $d_i=[d_{i1},\cdots,d_{ik}, \cdots, d_{iK}] \in \mathbb{R}_+^K$, where $d_{ik}=||\theta_i-\theta_k||$ is the distance between position vectors of agent $i$, $\theta_i \in \mathbb{R}^2$ and target $k$, $\theta_k \in \mathbb{R}^2$ in a 2D plane. According to the utility function above, agent $i$ only receives a positive payoff by selecting target $k$, if no other agent selects the target $k$. In this case, the payoff agent $i$ receives is inversely proportional to the distance of the agent to the target selected. Hence, no agent has more utility by changing its target, if other agents cover remaining targets. The payoff function ensures that any action profile that is one-to-one assignment between agents and targets is a NE. 

Agents' positions are sampled from identical and independent normal distributions with mean $0$ and variance $0.1$ for each axis of the positions. Targets are positioned around a circle centered at origin with radius $1$ and equal distances to each other. The target positions are unknown. Each agent receives private signals $\vartheta^i_{kt}$ at each time step $t$ about the position of each target $k$. The private signals $\vartheta^i_{t}=[\vartheta^i_1,\cdots,\vartheta^i_K]^T $ for each agent $i$ come from a multivariate normal distribution with mean $\theta=[\theta_1,\cdots,\theta_K]^T$ and  covariance matrix $\sigma I$, where $\sigma=0.5$ and $I \in \mathbb{R}^{K \times K}$ is the identity matrix. We assume agents receive signals up until time $\tau=10$. At time $\tau$, agent $i$'s point estimate of target $k$'s position is given by $\hat{\theta}_k=(1/ \tau)\sum_{t=1}^{\tau} \vartheta^i_{t}$. Final time $T_f$ is set to $50$. If the beliefs and distances were identical, the game with payoffs in \eqref{util_target} is a potential game. The existence of different beliefs and different distances to targets creates a near-potential game.




For numerical experiments, we use ring and stars networks. In the ring network, each agent $i\not = N$ shares its information only with agent $i+1$, and $i=N$ sends to $i=1$. In star network, there exists a central agent, which gathers and distributes information from all other agents. Network weights are set as $w_{i,l}^i=0.75$ and $w^i_{j,l}=0.25/|\ccalN_{i,t}|, \forall{j} \in \ccalN_{i,t}$, so that $\sum_{l \in \ccalN_{i,t} \cup \{i\} } w_{jl,t}^i=1$. 



Given this setup, we generate $20$ replications with random initial positions and samples of beliefs for each network type. In both cases, the final joint action profile $a_{T_f}$ is equal to a pure NE $a^*$ in all $20$ replications. This validates Theorem \ref{thm_min} and Corollary \ref{cor_1} since converging to a pure NE assures being an element of the given set $C_{N\delta}$. Fig.~\ref{fig_conv}(Left) verifies the convergence rate $O(\log t/ t)$ of estimation error as it goes to $0$. Fig.~\ref{fig_conv}(Middle) shows the average rate of convergence of the empirical frequencies to a NE of the target assignment game.  Agents learn each others' empirical frequencies faster in the star network.  Fig.~\ref{fig_conv}(Right) indicates the action selection of agents over time for a given run. Around $T_f/2=25$, agents reach a NE. 

\section{Conclusion}
We studied the convergence properties of DFP in near-potential games when the communication network is time-varying. We showed empirical frequencies of actions converge to a set of strategies with potential function values that are large enough compared to approximate Nash equilibria. These results established that convergence properties of DFP are identical to that of FP in near-potential games.
\appendix

\begin{lemma}\label{lem_sigma_ab}
Let  $(\sigma',\sigma'') \in \Delta \ccalA^N \times \Delta \ccalA^N$ be mixed joint action profiles such that $|| \sigma'-\sigma''|| \le \xi $, for a small enough $\xi>0$. If $\sigma' \in \Delta_{\alpha}$, then $\sigma'' \in \Delta_{\alpha+\beta}$, where $\Delta_{\alpha}$ and $\Delta_{\alpha+\beta}$ are $\alpha-$NE and $\alpha+\beta-$NE sets with respect to the order given $\alpha\ge 0$ and $\beta \ge 0$.
\end{lemma}
\begin{proof}
To define an approximate $\epsilon-$NE set, let $\psi:\Delta \ccalA^N \rightarrow \mathbb{R}$ be a function as follows,
\begin{equation} \label{eq_psi}
    \psi(\sigma)=-\underset{i \in \mathcal{N}, \,  a_i \in \ccalA}{\max} (u_i(a_i,\sigma_{-i})-u_i(\sigma_i,\sigma_{-i})).
\end{equation}
Then, it holds that a mixed joint action $\sigma=(\sigma_i,\sigma_{-i})$ is an $\epsilon-$NE, if and only if $\psi(\sigma) \ge -\epsilon$. By Lipschitz continuity of the mixed extension of the utility function, the difference $u_i(a_i,\sigma_{-i})-u_i(\sigma_i,\sigma_{-i})$ is Lipschitz continuous. Using the fact that $\psi(\sigma)$ is defined as maximum over finite set of such differences, $\psi(\sigma)$ is also Lipschitz continuous. Hence, there exists a Lipschitz constant $L_{\psi} \in \mathbb{R}^+$ such that for any $|| \sigma'-\sigma''|| \le \xi \le \beta / L_{\psi}$, it holds, $|\psi( \sigma')-\psi(\sigma'')| \le \beta$. Thus, by the definitions of $\psi:\Delta \ccalA^N \rightarrow \mathbb{R}$ and $\epsilon-$NE in \eqref{eq_app_Nash}, if $\sigma' \in \Delta_{\alpha}$, then $\sigma'' \in \Delta_{\alpha+\beta}$. 
\end{proof}

We use the notion of an upper semi-continuous correspondence in the next set of results.

\begin{definition}[Upper Semi-continuous Correspondence]\label{def_upsc_cr}
A correspondence $h: X \Rightarrow Y$ is upper semi-continuous, if one of the following statements hold,
\begin{itemize}
    \item For any $\bar{x} \in X$ and any open neighborhood $V$ of $h(\bar{x})$, there exists a neighborhood $U$ of $\bar{x}$, such that $h(x) \subset V$, and $h(x)$ is a compact set for all $x \in U$.
    \item $Y$ is compact, and the set, \textit{i.e.}, its graph, $\{(x,y)| x \in X, y \in h(x)\}$ is closed. 
\end{itemize}
\end{definition}
\begin{lemma} \label{lem_up_sc}
Let $h:\mathbb{R} \Rightarrow \Delta \ccalA^N $ be the correspondence representing the set of $\alpha$-NE strategies, i.e.,
\begin{equation}\label{eq_usc}
    h(\alpha)=\Delta_{\alpha}=\{ \sigma \in \Delta \ccalA^N | \psi(\sigma) \ge -\alpha\}
\end{equation}
where $\psi$ is defined in \eqref{eq_psi}.
Then, the correspondence $h:\mathbb{R} \Rightarrow \Delta \ccalA^N  $ is upper semi-continuous. 
\end{lemma}
\begin{proof}
Since $\psi$ is a (Lipschitz) continuous function from Lemma \ref{lem_sigma_ab}, the set $h(\alpha)$ is closed by Proposition 1.1.2 in \cite[Ch. 1]{bertsekas2009convex}. As a result, the graph of the correspondence $h$ is closed. Hence, $\Delta_{\alpha}$ is compact as $\Delta_{\alpha} \subseteq \Delta \ccalA^N$ is bounded, which satisfies Definition \ref{def_upsc_cr}.
\end{proof}

\begin{lemma}\label{lem_subset_c}
Let $C_{N\delta}$ and $C_{N\delta+1/q}$ for $q \in \mathbb{Z}^+$ be the closed sets defined as in \eqref{eq_near_eq}. Further, we define the set $S_{\xi}$ for $\xi>0$ as follows, 
\begin{equation}\label{eq_xi_neighborhood}
    S_{\xi}:= \{ \sigma \in \Delta \ccalA^N \: | \: \underset{y \in C_{N\delta}}{\min} ||\sigma-y|| < \xi\}.
\end{equation}
Then, for any $\xi>0$ and for some $q \in \mathbb{Z}^+$, it holds that $C_{N\delta+1/q} \subset S_{\xi}$.
\end{lemma}
\begin{proof}
 Using the properties of upper semi-continuity (Definition \ref{def_upsc_cr}), there exists $\zeta>0$ such that the set $\{ \sigma \in \Delta \ccalA^N \: | \: u(\sigma) \ge  \underset{y \in \Delta_{N\delta}}{\min} u(y)-\zeta\} $ is a subset of $\xi$-neighborhood of the set $C_{N \delta}$, i.e., $\{ \sigma \in \Delta \ccalA^N \: | \: u(\sigma) \ge  \underset{y \in \Delta_{N\delta}}{\min} u(y)-\zeta\} \subseteq S_\xi$.
Next, since the correspondence $h$ \eqref{eq_usc} is upper semi-continuous by Lemma \ref{lem_up_sc}, for any $\xi'>0$ there exists a large enough $q$ such that $\Delta_{N\delta+1/q}$ is contained in $\xi'$ neighborhood of $\Delta_{N\delta}$. That is, for any $\xi'>0$, for any point $z \in \Delta_{N\delta+1/q}$, there exists $y \in \Delta_{N\delta}$ such that $||z-y|| \le \xi'$. Given that the mixed extension of the potential function $u$ is Lipschitz continuous with Lipschitz constant $L$, and defining $\zeta/L\ge \xi' >0$, the following holds for large enough $q$,
\begin{equation}\label{eq_lem_ck}
    \underset{z \in \Delta_{N\delta+1/q}}{\min} u(z) \ge  \underset{y \in \Delta_{N\delta}}{\min} u(y)-\zeta.
\end{equation}
Hence, for any $z \in C_{N\delta+1/q}$, the inequality  \eqref{eq_lem_ck} is satisfied. Thus, for large enough $q$ and any $z \in C_{N\delta+1/q}$, it also holds $z \in S_{\xi}$ and $C_{N\delta+1/q} \subset S_{\xi}$.
\end{proof}
\bibliographystyle{IEEEtran}
\bibliography{bibliography}

\end{document}